\documentclass[11pt]{amsart}

\usepackage[margin=1in]{geometry}
\usepackage{amsmath,amssymb,amsfonts,amsthm,mathtools}
\usepackage{mathrsfs}
\usepackage{enumitem}
\usepackage{hyperref}

\hypersetup{
  colorlinks=true,
  linkcolor=blue,
  citecolor=green,
  urlcolor=blue
}

\newtheorem{thm}{Theorem}[section]
\newtheorem{lem}[thm]{Lemma}
\newtheorem{prop}[thm]{Proposition}

\theoremstyle{definition}
\newtheorem{defn}[thm]{Definition}
\newtheorem{rmk}[thm]{Remark}
\newtheorem{ques}[thm]{Question} 

\newcommand{\ZZ}{\mathbb{Z}}

\newcommand{\RR}{\mathbb{R}}

\begin{document}

\title{Co-Hopfianity is not a profinite property}

\author{Hyungryul Baik}
\address{Department of Mathematical Sciences, KAIST,
291 Daehak-ro, Yuseong-gu, 34141 Daejeon, South Korea}
\email{hrbaik@kaist.ac.kr}

\author{Wonyong Jang}
\address{Department of Mathematical Sciences, KAIST,
291 Daehak-ro, Yuseong-gu, 34141 Daejeon, South Korea}
\email{jangwy@kaist.ac.kr}

\maketitle

\begin{abstract}
We exhibit two finitely generated residually finite groups $G$ and $H$ with isomorphic profinite completions $\widehat{G} \cong \widehat{H}$, such that $G$ is co-Hopfian while $H$ is not.
The construction utilizes Wise's residually finite version of the Rips construction applied to a finitely presented acyclic group $U$ with trivial profinite completion and a strong universality property.
A key feature of our approach is the construction of $H$ as a preimage subgroup of $G$ which is \emph{conjugate to a proper subgroup of itself}.
This renders the non-co-Hopfianity of $H$ immediate without requiring a detailed structural analysis of the Rips kernel.
\end{abstract}

\vspace{0.6cm}
\noindent \textbf{AMS Classification numbers (2020)} Primary: 20E18, 20F67; Secondary: 20E06, 20E26

\vspace{0.6cm}

\section{Introduction}\label{sec:intro}

The classification of finitely generated residually finite groups by their profinite completions is a central theme in geometric group theory.
Since the profinite completion $\widehat{\Gamma}$ of a group $\Gamma$ encodes all finite quotients of $\Gamma$, it is natural to ask to what extent $\widehat{\Gamma}$ determines algebraic or geometric properties of $\Gamma$.
This line of inquiry, initiated by Grothendieck, has led to the study of \emph{profinite rigidity}.
A group property $\mathcal{P}$ is said to be a \emph{profinite property} if for any two finitely generated residually finite groups $\Gamma_1$ and $\Gamma_2$ with $\widehat{\Gamma}_1 \cong \widehat{\Gamma}_2$, the group $\Gamma_1$ has $\mathcal{P}$ if and only if $\Gamma_2$ has $\mathcal{P}$.

While some properties, such as being abelian, being nilpotent, satisfying a nontrivial law \cite{bridson2025chasing}, and being polycyclic \cite{sabbagh1991polycyclic}, are profinite properties, it is well known that many geometric and asymptotic properties are not.
Notable examples of properties that fail to be profinite include amenability \cite{kionke2023amenability}, 
Kazhdan’s property (T) \cite{aka2012profinite}, 
being conjugacy separable \cite{cotton2012conjugacy}, 
various finiteness properties \cite{lubotzky2014finiteness}, 
Serre’s property FA \cite{cheetham2024property}, 
bounded cohomology \cite{echtler2024bounded}, 
being torsion-free \cite{bridson2024profinite}, 
left-orderability \cite{bridson2024profinite}, and 
bi-orderability \cite{jang2025bi}.

In this paper, we focus on the \emph{co-Hopfian} property.
A group is co-Hopfian if it does not contain a proper subgroup isomorphic to itself.
This property can be viewed as a dual notion to the \emph{Hopfian} property, which requires that every surjective homomorphism to itself is an isomorphism.
As finitely generated residually finite groups are Hopfian (by Mal'cev \cite{MR3420}), Hopficity is vacuously a profinite invariant among them.
However, if we take the definition of a profinite invariant to just mean ``detectable by finite quotients'', so that it does not \emph{a priori} assume residual finiteness, we can find easy counterexamples: one can take any finitely generated non-Hopfian group and its largest residually finite quotient, like the Baumslag--Solitar group $\mathrm{BS}(2, 3)$ and $\ZZ[1/6] \rtimes_{2/3} \ZZ$. While perhaps not that interesting, it is still worth mentioning. In contrast, the co-Hopfian property is far more delicate.

Our main goal is to demonstrate that co-Hopfianity is \emph{not} a profinite property.

\begin{thm}[Theorem \ref{thm:main}] \label{thm:intro-main}
There exist finitely generated residually finite groups $G$ and $H$ such that:
\begin{enumerate}[label=(\arabic*)]
  \item $\widehat{G} \cong \widehat{H}$;
  \item $G$ is co-Hopfian;
  \item $H$ is not co-Hopfian.
\end{enumerate}
\end{thm}

Our construction relies on Wise's residually finite version of the Rips construction \cite{wise2003residually}, combined with a finitely presented acyclic group $U$ with trivial profinite completion and a universality property established by Bridson \cite{bridson2019homology}.
We construct $G$ as a torsion-free hyperbolic, residually finite group surjecting onto $U$ with finitely generated kernel $K$.
We then choose a subgroup $A<U$ and an element $t\in U$ such that $t^{-1}At\subsetneq A$ and $A\cong U$.
Defining $H:=\pi^{-1}(A)<G$, conjugation by a lift of $t$ gives an isomorphism from $H$ onto a proper subgroup of itself, so $H$ is not co-Hopfian.
Finally, we show $\widehat{G}\cong\widehat{H}$ using a criterion of Bridson--Grunewald \cite{bridson2010schur}.

The paper is organized as follows. 
In Section~\ref{sec:prelim}, we briefly recall the definition of co-Hopfian groups, profinite completions, and Rips construction.
We will then describe the group constructed by Bridson.
Using Rips construction and Bridson's group, we will prove that toral relative hyperbolicity is also not a profinite property (Proposition~\ref{prop:Toral_RH}).
We will complete the proof of the main result, Theorem~\ref{thm:main}, in Section~\ref{sec:main}.

\vspace{0.4cm}

\noindent \textbf{Acknowledgment.}
Both authors thank Martin Bridson for fruitful discussions and for his interest in this work.
We also thank Jan Kim and Francesco Fournier-Facio for useful discussions and suggestions which improved the exposition of the paper. 
This work was supported by the National Research Foundation of Korea (NRF) grant funded by the Korean government (MSIT)
(No.\ RS-2025-00513595).

\section{Preliminaries}\label{sec:prelim}

\subsection{Hopfian and co-Hopfian groups}
A group $G$ is \emph{co-Hopfian} if every injective homomorphism $\phi:G\to G$ is surjective.
Equivalently, $G$ is not isomorphic to any of its proper subgroups.
Conversely, $G$ is \emph{Hopfian} if every surjective homomorphism $G\to G$ is injective.
Recall that finitely generated residually finite groups are Hopfian, but they need not be co-Hopfian.

As mentioned in the Introduction, the co-Hopfian property is more delicate than the Hopfian property. 
It is known that every hyperbolic group is Hopfian (the torsion-free case was previously obtained by Sela \cite{sela1999hopf}, and the general case by Weidmann--Reinfeldt \cite[Corollary 7.9]{weidmann2019makanin}), but a free group is not co-Hopfian. Hence, there are many hyperbolic groups that are not co-Hopfian.
We will use the following result, which was originally proved by Z. Sela for the torsion-free case \cite[Theorem 4.4]{sela1997structure} and extended to all one-ended hyperbolic groups without the torsion-free assumption by C. Moioli \cite{moioli2013thesis}.

\begin{thm}[{\cite[Theorem 4.4]{sela1997structure}, \cite{moioli2013thesis}}] \label{thm:Sela}
Every one-ended hyperbolic group is co-Hopfian.
\end{thm}

\subsection{Profinite completion}
 One of the main concepts in this paper is that of profinite properties.
 To introduce this notion precisely, we begin with the definition of profinite groups.

\begin{defn}
 We say that a group $G$ is \emph{profinite} if $G$ is a topological group isomorphic to the inverse limit of an inverse system of discrete finite groups.
\end{defn}
 
For a group $G$, let $\mathcal{N}$ be the set of all normal subgroups of finite index in $G$.
Consider the following inverse limit
\[
\widehat{G}=\varprojlim_{N\in\mathcal{N}} G/N.
\]
By the definition, $\widehat{G}$ is profinite. The group $\widehat{G}$ is called the \emph{profinite completion} of $G$.

\begin{rmk}
 There is the natural homomorphism $G$ to its profinite completion $\widehat{G}$ given by \[ i : g \mapsto \{ gN \}_{N \in \mathcal{N}}. \]
This homomorphism is not injective in general. Indeed, the homomorphism $i$ is injective if and only if $G$ is residually finite.
\end{rmk}

Let $\mathcal{C}(G)$ be the set of isomorphism classes of all finite quotients of $G$.
For finitely generated groups $G$ and $H$, $\widehat{G} \cong \widehat{H}$ if and only if $\mathcal{C}(G) = \mathcal{C}(H)$ (see \cite{dixon1982profinite} and \cite{nikolov2007finitely}).
In other words, $G$ and $H$ have the isomorphic profinite completion if and only if they have the same finite quotients up to isomorphic class.

We now define the concept of a profinite property using the profinite completion.

\begin{defn}
 Let $\mathcal{P}$ be a group property of finitely generated residually finite groups. We say that $\mathcal{P}$ is a \emph{profinite property} if for any two finitely generated residually finite groups $G_1$ and $G_2$ with $\widehat{G_1} \cong \widehat{G_2}$, $G_1$ satisfies $\mathcal{P}$ if and only if $G_2$ satisfies $\mathcal{P}$.
\end{defn}

Namely, if a group property $\mathcal{P}$ is a profinite property, then it can be detected by profinite completion, and equivalently, by taking finite quotients.
For more details on profinite groups, profinite rigidity, and profinite properties, 
we refer to \cite{nikolov2007finitely}, \cite{MR2599132}, \cite{reid2018profinite}, and \cite{bridson2025chasing}.

We close this subsection by recalling the following fact. This result gives us the isomorphism between their profinite completions.

\begin{lem}[{\cite[Theorem A]{bridson2010schur}}]\label{lem:Profinite_Completion}
Let
\[
1 \to N \to G \to Q \to 1
\]
be a short exact sequence where $Q$ is finitely presented.
Then the map $\widehat{N}\to\widehat{G}$ induced by inclusion is an isomorphism if $\widehat{Q}=1$ and $H_2(Q,\ZZ)=0$.
\end{lem}

\begin{rmk}
The statement in \cite{bridson2010schur} is formulated in a slightly more general way; in the case $\widehat{Q}=1$, the relevant condition reduces to $H_2(Q,\ZZ)=0$.
\end{rmk}

\subsection{Wise's residually finite Rips construction}
In \cite{MR642423}, E.~Rips suggested the construction of a hyperbolic group from an arbitrary finitely presented group.
Using this technique, a lot of subgroups of hyperbolic groups with a pathological property are established.

 This tool was revised in many literatures for various purposes.
In this paper, we recall the following variation of the Rips construction due to Wise.
This revised technique guarantees that a hyperbolic group $G$ is residually finite.

\begin{thm}[{\cite[Theorem 3.1]{wise2003residually}}]\label{thm:RF_Rips_Const}
Let $Q$ be a finitely presented group.
Then there exists a short exact sequence
\[
1 \to K \to G \xrightarrow{\pi} Q \to 1
\]
such that:
\begin{itemize}
  \item $K$ is finitely generated (indeed, generated by three elements),
  \item $G$ is hyperbolic,
  \item $G$ is torsion-free,
  \item $G$ is residually finite.
\end{itemize}
\end{thm}

\subsection{A universal acyclic group}
We will use the following group constructed by Bridson in \cite{bridson2019homology}.

\begin{lem}[{\cite[Theorem A]{bridson2019homology}}]\label{lem:U}
There exists a finitely presented group $U$ with the following properties:
\begin{enumerate}[label=(\alph*)]
  \item $U$ is acyclic, i.e.\ $H_i(U,\ZZ)=0$ for all $i\ge 1$;
  \item $\widehat{U}=1$;
  \item every finitely presented group embeds in $U$.
\end{enumerate}
\end{lem}

Recall that to disprove a group property $\mathcal{P}$ is a profinite property, we construct two finitely generated residually finite groups $G$ and $H$ with $\widehat{G} \cong \widehat{H}$ such that $G$ has $\mathcal{P}$ but $H$ does not.
In Wise's result, $G$ and $K$ are finitely generated and residually finite.
Moreover, when we choose $Q=U$, the group in Lemma~\ref{lem:U}, we obtain
$\widehat{G} \cong \widehat{K}$ by Lemma~\ref{lem:Profinite_Completion}.
Therefore, when we find a group property $\mathcal{P}$ such that $G$ has $\mathcal{P}$ but $K$ does not, or vice versa,
it immediately follows that $\mathcal{P}$ is not a profinite property.
In particular, we can prove that toral relative hyperbolicity is not a profinite property.
Recall that a group $G$ is \emph{toral relatively hyperbolic} if $G$ is torsion-free and is hyperbolic relatively to finitely many finitely generated abelian subgroups.
The class of toral relativively hyperbolic groups contains torsion-free hyperbolic groups, limit groups (\cite{alibegovic2005combination} and \cite{dahmani2003combination}), and groups acting freely on $\RR^n$-trees \cite{guirardel2004limit}.

\begin{prop} \label{prop:Toral_RH}
Toral relative hyperbolicity is not a profinite property.
\end{prop}

\begin{proof}
Let $U$ be the acyclic group from Lemma~\ref{lem:U}. By applying Theorem~\ref{thm:RF_Rips_Const} with $Q=U$, we obtain the following short exact sequence:
\[
1 \to K \to G \to U \to 1.
\]
By the construction, both $G$ and $K$ are finitely generated and residually finite. Furthermore, Lemma~\ref{lem:Profinite_Completion} implies that $\widehat{G} \cong \widehat{K}$.
Since $G$ is a torsion-free non-elementary hyperbolic group, $G$ is toral relatively hyperbolic.

Now, assume for the sake of contradiction that $K$ is also toral relatively hyperbolic.
This implies that $K$ is hyperbolic relative to a collection of subgroups which are virtually abelian.
Since $K$ is a subgroup of the hyperbolic group $G$, it cannot contain $\ZZ^2$.
Consequently, if $K$ were toral relatively hyperbolic, it would be hyperbolic relative to finite subgroups or trivial subgroups, which implies that $K$ itself must be hyperbolic.
However, it is a standard property of the Rips construction that if the quotient group $U$ is infinite, the kernel $K$ is not finitely presented (see \cite{wise2003residually} or \cite{MR642423}).
On the other hand, if $K$ were hyperbolic, it would necessarily be finitely presented.
This yields a contradiction.

Therefore, $K$ cannot be toral relatively hyperbolic, whereas $G$ is. This completes the proof.
\end{proof}

\begin{rmk}
It is well established that hyperbolicity is not a profinite property.
However, to the best of the authors' knowledge, the question of whether toral relative hyperbolicity constitutes a profinite property has not been explicitly addressed in the existing literature.
We note that, in the specific context of virtually compact special groups, toral relative hyperbolicity is indeed a profinite property, as detected by the profinite completion due to a recent result of P. Zalesskii \cite[Theorem 3.7]{zalesskii2024profinite}.
Our result demonstrates that this rigidity does not hold for general finitely generated residually finite groups.
\end{rmk}

\section{The proof of the main result}\label{sec:main}

We now prove the following.

\begin{thm} \label{thm:main}
There exist finitely generated residually finite groups $G$ and $H$ satisfying:
\begin{enumerate}[label=(\arabic*)]
  \item $\widehat{G}\cong \widehat{H}$;
  \item $G$ is co-Hopfian;
  \item $H$ is not co-Hopfian.
\end{enumerate}
Consequently, co-Hopfianity is not a profinite property.
\end{thm}

\subsection{The construction of a co-Hopfian group $G$}

Let $U$ be as in Lemma~\ref{lem:U}.
Apply Theorem~\ref{thm:RF_Rips_Const} with $Q=U$ to obtain
\[
1\to K \to G \xrightarrow{\pi} U \to 1.
\]

\begin{lem}\label{lem:G-coHopfian}
The group $G$ is co-Hopfian.
\end{lem}

\begin{proof}
By Theorem~\ref{thm:Sela}, a one-ended hyperbolic group is co-Hopfian.
The group $G$ is a torsion-free hyperbolic group by Theorem~\ref{thm:RF_Rips_Const}; it remains to show that $G$ is one-ended.

Suppose $G$ is not one-ended. Since $G$ is finitely generated, Stallings' theorem on ends implies that $G$ splits nontrivially as a free product
\[
G = G_1 * G_2
\]
with $G_1,G_2$ nontrivial.
It is a classical fact that if $G=G_1*G_2$ is a nontrivial free product and $N\lhd G$ is a nontrivial finitely generated normal subgroup, then $N$ has finite index in $G$ (see \cite[p.~679]{baumslag1966intersections} and \cite{karrass1970subgroups}).
Applying this to the nontrivial finitely generated normal subgroup $K\lhd G$, we would conclude that $[G:K]<\infty$.
This contradicts $G/K\cong U$ being infinite.
Therefore, $G$ is one-ended, and hence co-Hopfian.
\end{proof}

\subsection{A self-embedding via conjugation inside $U$}

\begin{lem} \label{lem:U-structure}
Let $U$ be the group in Lemma~\ref{lem:U}.
There exists a subgroup $A<U$ and an element $t\in U$ such that:
\begin{enumerate}[label=(\arabic*)]
  \item $A \cong U$;
  \item $t^{-1}At \subsetneq A$.
\end{enumerate}
\end{lem}

\begin{proof}
Let $F_2$ be the free group of rank $2$. The free product $U*F_2$ is finitely presented.
By Lemma~\ref{lem:U}(c), there exists an embedding
\[
\psi: U*F_2 \hookrightarrow U.
\]
We claim that $\psi$ is not surjective.
Indeed, since $U$ is acyclic, $H_1(U;\ZZ)=0$.
On the other hand,
\[
H_1(U*F_2;\ZZ)\cong H_1(U;\ZZ)\oplus H_1(F_2;\ZZ)\cong 0\oplus \ZZ^2 \cong \ZZ^2.
\]
If $\psi$ were surjective, it would be an isomorphism, forcing $H_1(U;\ZZ)\cong H_1(U*F_2;\ZZ)$, a contradiction.
Thus $\psi(U*F_2)$ is a proper subgroup of $U$.

Let $i:U \hookrightarrow U*F_2$ be the natural inclusion of the free factor.
Define
\[
\sigma := \psi\circ i : U \to U.
\]
Then $\sigma$ is injective, and its image satisfies
\[
\sigma(U)=\psi(i(U)) \subset \psi(U*F_2) \subsetneq U,
\]
so $\sigma$ is not surjective.

Now form the ascending HNN extension
\[
L \;:=\; \langle\, S_U, s \mid R_U, s^{-1}us = \sigma(u)\ \forall\, u\in U\,\rangle.
\]
Here, $\left< S_U \mid R_U \right>$ is a group presentation for $U$.
Since $U$ is finitely presented and $\sigma$ is determined by images of generators, the group $L$ is finitely presented.
By Lemma~\ref{lem:U}(c) again, there exists an embedding
\[
\theta: L \hookrightarrow U.
\]
Set $A:=\theta(U)$ and $t:=\theta(s)$.
Then $A\cong U$ since $\theta$ is injective.
Inside $L$ we have
\[
s^{-1}Us = \sigma(U) \subsetneq U,
\]
and applying $\theta$ yields
\[
t^{-1}At \;=\; \theta(s)^{-1}\theta(U)\theta(s)
\;=\; \theta(s^{-1}Us)
\;=\; \theta(\sigma(U))
\;\subsetneq\; \theta(U)=A.
\]
This completes the proof.
\end{proof}

\subsection{Constructing $H$ as a preimage subgroup of $G$}

Let $A<U$ and $t\in U$ be as in Lemma~\ref{lem:U-structure}.
Recall the short exact sequence from Wise's construction:
\[
1 \to K \to G \xrightarrow{\pi} U \to 1.
\]
Define
\[
H := \pi^{-1}(A) < G.
\]
Choose $g\in G$ with $\pi(g)=t$.

\begin{lem}\label{lem:conj-preimage}
Let $c_g:G\to G$ be conjugation by $g$, $c_g(x)=g^{-1}xg$.
Then
\[
c_g(H)=\pi^{-1}(t^{-1}At).
\]
\end{lem}

\begin{proof}
If $\alpha\in c_g(H)$, then $\alpha=g^{-1}hg$ for some $h\in H$, so
\[
\pi(\alpha)=\pi(g)^{-1}\pi(h)\pi(g)=t^{-1}\pi(h)t \in t^{-1}At.
\]
Hence $\alpha\in \pi^{-1}(t^{-1}At)$.

Conversely, if $\beta\in \pi^{-1}(t^{-1}At)$, then $\pi(\beta)\in t^{-1}At$, so
\[
\pi(g\beta g^{-1})=\pi(g)\pi(\beta)\pi(g)^{-1}\in A,
\]
hence $g\beta g^{-1}\in H$ and therefore $\beta\in g^{-1}Hg=c_g(H)$.
\end{proof}

\begin{prop}\label{prop:H-not-cohopf}
The group $H$ is not co-Hopfian.
\end{prop}

\begin{proof}
By Lemma~\ref{lem:U-structure}, we have $t^{-1}At\subsetneq A$.
Taking full preimages under $\pi$ preserves strict inclusion, so
\[
\pi^{-1}(t^{-1}At)\subsetneq \pi^{-1}(A)=H.
\]
By Lemma~\ref{lem:conj-preimage}, $\pi^{-1}(t^{-1}At)=c_g(H)$.
Conjugation restricts to an isomorphism $c_g|_H:H\to c_g(H)$, and $c_g(H)$ is a proper subgroup of $H$.
Thus $H$ is isomorphic to a proper subgroup of itself, so $H$ is not co-Hopfian.
\end{proof}

\subsection{Profinite completions and finiteness properties}

\begin{lem}\label{lem:samePC}
$\widehat{G}\cong \widehat{H}$.
\end{lem}

\begin{proof}
First consider $1\to K\to G\to U\to 1$.
By Lemma~\ref{lem:U}, $\widehat{U}=1$ and $H_2(U;\ZZ)=0$ since $U$ is acyclic.
By Lemma~\ref{lem:Profinite_Completion}, the inclusion-induced map $\widehat{K}\to\widehat{G}$ is an isomorphism.

Next, since $H=\pi^{-1}(A)\supset K$, we have a short exact sequence
\[
1\to K \to H \to A \to 1.
\]
Because $A\cong U$ by Lemma~\ref{lem:U-structure}(1), we have $\widehat{A}\cong\widehat{U}=1$ and $H_2(A;\ZZ)\cong H_2(U;\ZZ)=0$.
Again by Lemma~\ref{lem:Profinite_Completion}, the inclusion-induced map $\widehat{K}\to\widehat{H}$ is an isomorphism.
Therefore
\[
\widehat{G}\cong \widehat{K}\cong \widehat{H}.
\]
\end{proof}

\begin{lem}\label{lem:fg-rf}
The groups $G$ and $H$ are finitely generated and residually finite.
\end{lem}

\begin{proof}
By Theorem~\ref{thm:RF_Rips_Const}, $G$ is finitely generated and residually finite.
Since $H<G$, $H$ is also residually finite.

To see $H$ is finitely generated, consider the short exact sequence
\[
1\to K \to H \to A \to 1.
\]
Here $K$ is finitely generated (Theorem~\ref{thm:RF_Rips_Const}) and $A\cong U$ is finitely presented (hence finitely generated).
An extension of a finitely generated group by a finitely generated group is finitely generated; thus $H$ is finitely generated.
\end{proof}

\begin{proof}[Proof of Theorem~\ref{thm:main}]
By Lemma~\ref{lem:fg-rf}, $G$ and $H$ are finitely generated residually finite.
Lemma~\ref{lem:samePC} shows $\widehat{G}\cong \widehat{H}$.
Lemma~\ref{lem:G-coHopfian} shows $G$ is co-Hopfian.
Proposition~\ref{prop:H-not-cohopf} shows $H$ is not co-Hopfian.
Hence, co-Hopfianity is not a profinite property.
\end{proof}

In our construction, $G$ is hyperbolic and hence finitely presented.
However, we do not know whether $H$ is finitely presented.
Although we obtain a finitely generated example, it remains open whether such an example can be finitely presented.
We therefore pose the following question.

\begin{ques}
Is co-Hopfianity a profinite property among finitely presented residually finite groups?
\end{ques}


\end{document}